\documentclass[reqno,twoside,11pt]{amsart}

\usepackage{dsfont}

\usepackage{verbatim}

\theoremstyle{plain}

\numberwithin{equation}{section}

\usepackage{graphicx}
\usepackage{amssymb}
\usepackage{amsmath}
\usepackage{latexsym}
\usepackage{mathrsfs}
\usepackage{textcomp}
\usepackage{fancyhdr}
\usepackage{pxfonts}
\usepackage{esint}
\usepackage{textcomp}
\usepackage{vmargin}

\newcommand{\norm}[1]{\left\lVert #1 \right\rVert}

\newcommand{\dx}{\;\mbox{d}x}
\newcommand{\dt}{\;\mbox{d}t}
\newcommand{\ds}{\;\mbox{d}s}
\newcommand{\osc}[1]{\underset{#1}{\mathrm{osc}}\,}

\theoremstyle{definition}

\newtheorem{Teo}{Theorem}[section]

\newtheorem{Lemma}[Teo]{Lemma}

\newtheorem{Prop}[Teo]{Proposition}

\newtheorem{theorem}{Theorem}[section]

\DeclareRobustCommand{\gobblefour}[4]{}

\begin{document}

\title[$p$-Orthotropic functions in the plane]
{Regularity of the derivatives of $p$-orthotropic functions in the plane for $1<p<2$}

\author{Diego Ricciotti}
\address{Diego Ricciotti, University of Pittsburgh, Department of Mathematics, 
301 Thackeray Hall, Pittsburgh, PA 15260, USA}
\email{DIR17@pitt.edu}

\begin{abstract}
We present a proof of the $C^1$ regularity of $p$-orthotropic functions in the plane for $1<p<2$, based on the monotonicity of the derivatives. 
Moreover we achieve an explicit logarithmic modulus of continuity. 
\end{abstract}
\maketitle
\tableofcontents

\let\thefootnote\relax
\let\thefootnote\relax
\section{Introduction}


In this work we investigate the regularity of $p$-orthotropic functions in the plane for $1<p<2$.
Let $\Omega\subset\mathbb{R}^2$ be an open set. A weak solution of the orthotropic $p$-Laplace equation (also known as pseudo $p$-Laplace equation) is a function $u\in W^{1,p}(\Omega)$ such that
\begin{equation}\label{orthodeg}
\sum_{i=1}^2 \int_\Omega |\partial_i u|^{p-2} \partial_i u \, \partial_i \phi \dx=0 \quad \text{for all} \quad \phi\in W_0^{1,p}(\Omega).
\end{equation}
Equation \eqref{orthodeg} arises as the Euler-Lagrange equation for the functional
\begin{equation}
I_{\Omega}(v) = \sum_{i=1}^2\int_{\Omega} \frac{|\partial_i v|^p}{p}\dx.
\end{equation}
The equation is singular when either one of the derivatives vanishes, and does not fall into the category of equations with $p$-Laplacian structure. 
It was proved by Bousquet and Brasco in \cite{BB} that weak solutions of \eqref{orthodeg} for $1<p<\infty$ are $C^1(\Omega)$. A simple proof which gives a logarithmic modulus of continuity for the derivatives is contained in \cite{LR} for the case $p\geq 2$. The latter relies on a lemma on the oscillation of monotone functions due to Lebesgue \cite{L} and the fact that derivatives of solutions are monotone (in the sense of Lebesgue). The purpose of this work is to extend this result to the case $1<p<2$ employing methods developed in \cite{LR}. We obtain the following:
\begin{theorem}\label{C1}
Let $\Omega\subset\mathbb{R}^2$ and $u\in W^{1,p}(\Omega)$ be a solution of the  equation \eqref{orthodeg} for $1<p<2$. Fix a ball  $B_{R}\subset\subset \Omega$. Then,    
 for all $j\in\{1,2\}$ and $B_r\subset\subset B_{R/2}$, we have
\begin{equation}\label{oscEst2}
\osc{B_r} (\partial_ju)\leq C_p 
\left(\log\left(\frac{R}{r}\right)\right)^{-\frac{1}{2}}\left( \fint_{B_{R}}  |\nabla u|^p \dx\right)^\frac{1}{p},
\end{equation}
where $C_p$ is a constant depending only on $p$.
\end{theorem}

\paragraph{\bf Notation.}
We indicate balls by $B_r=B_r(a)=\{x\in\mathbb{R}^2\;:\; |x-a|<r\}$ and we omit the center when not relevant. Whenever two balls $B_r\subset B_R$ appear in a statement they are implicitly assumed to be concentric. The variable $x$ denotes the vector $(x_1, x_2)$ and we denote the partial derivatives of a function $f$ with respect to $x_j$ as $\partial_j f$.

\section{Regularization}
We will consider a regularized problem by introducing a non degeneracy parameter $\epsilon>0$. 
\\

Fix $B_R\subset\subset \Omega\subset \mathbb{R}^2$ and  consider  the regularized Dirichlet problem
\begin{equation}\label{orthonondeg2}
\begin{split}
\begin{cases}
\sum_{i=1}^2 \int_{B_R} (|\partial_i u^\epsilon|^2 +\epsilon)^\frac{p-2}{2} \partial_i u^\epsilon\, \partial_i \phi \dx=0\\
u^\epsilon-u \in W_0^{1,p}(B_R).
\end{cases}
\end{split}
\end{equation}
Note that $u^\epsilon$ is the unique minimizer of the regularized functional
\begin{equation}
I^\epsilon_{B_R}(v)=\sum_{i=1}^2\int_{B_R}\frac{1}{p}(|\partial_i v|^2+\epsilon)^\frac{p}{2}\dx
\end{equation}
among $W^{1,p}(B_R)$ functions $v$ such that $v-u\in W^{l,p}_0(B_R)$.
By elliptic regularity theory, the unique solution $u^\epsilon$ of \eqref{orthonondeg2} is smooth in $B_R$.\\
Fix an index $j\in\{1,2\}$. Then, replacing $\phi$ by $\partial_j\phi$ in equation \eqref{orthonondeg2} and integrating by parts, we find that the derivative $\partial_j u^\epsilon$ satisfies the following equation
\begin{equation}\label{orthoder2}
\sum_{i=1}^2 \int_{B_R} (\epsilon+|\partial_i u^\epsilon|^2)^\frac{p-4}{2} (\epsilon+(p-1)|\partial_i u^\epsilon|^2)\,  \partial_i \partial_j u^\epsilon \,\partial_i\phi \dx =0
\end{equation}
for all $\phi\in C_0^\infty(B_R)$.\\

We now collect some uniform estimates and convergences (see also \cite{BB}).

\begin{Lemma}
Let $u\in W^{1,p}(\Omega)$ be a solution of \eqref{orthodeg} and $u^\epsilon$ be a solution of \eqref{orthonondeg2} for $1<p<2$. Then we have
\begin{equation}\label{energy2}
\int_{B_R} |\nabla u^\epsilon|^p\dx \leq 
C_p \left(  \int_{B_{R}} |\nabla u|^p\dx +\epsilon^\frac{p}{2}R^2 \right)
\end{equation}
where $C_p$ is a constant depending only on $p$.
\end{Lemma}

\begin{proof}
The estimate follows from $$I^\epsilon_{B_R}(u^\epsilon)\leq I^\epsilon_{B_R}(u).$$
\end{proof}

\begin{Prop}
Let $u\in W^{1,p}(\Omega)$ be a solution of \eqref{orthodeg} and $u^\epsilon$ be a solution of \eqref{orthonondeg2} for $1<p<2$. Then, for all $j\in\{1,2\}$, we have
\begin{align}
\sup_{B_{R/2}} (\epsilon+|\nabla u^\epsilon|^2)&\leq C_p\left( \fint_{B_R}(\epsilon+|\nabla u^\epsilon|^2)^\frac{p}{2}\dx \right)^\frac{2}{p}, \label{lip2} \\ 
\int_{B_{R/2}}|\nabla \partial_j u^\epsilon|^2\dx &\leq C_p\left(\fint_{B_R}(|\nabla u|^p+\epsilon^\frac{p}{2})\dx\right)^\frac{2}{p} \label{grad2}
\end{align}
where $C_p$ is a constant depending only on $p$.
\end{Prop}

\begin{proof}
The proof of the Lipschitz bound can be found in \cite{FF}  while \eqref{grad2} appears in \cite{BB}. We provide details for completeness.
Note that by a change of variables, the function $u^\epsilon_R(x)=u^\epsilon(x_0+Rx)$ satisfies the equation 
\begin{equation}\label{orthonondeg2scaled}
\sum_{i=1}^2\int_{B_1}(|\partial_iu^\epsilon_R|^2+R^2\epsilon)^\frac{p-2}{2}\partial_iu^\epsilon_R\partial_i\phi\dx=0 \quad \text{for all} \quad \phi\in W_0^{1,p}(B_1).
\end{equation}
Introduce the notation $w=\epsilon R^2+|\nabla u^\epsilon_R|^2$ and $a_i(z)=a_i(z_i)=(\epsilon R^2+|z_i|^2)^\frac{p-2}{2}z_i$ so that equation \eqref{orthonondeg2scaled} rewrites as $$\sum_{i=1}^2\int_{B_1}a_i(\partial_iu^\epsilon_R)\partial_i\phi\dx=0 \quad \text{for all}\quad \phi\in W_0^{1,p}(B_1).$$ 
For $j\in\{1,2\}$ and $\alpha\geq 0$ take $\phi=\partial_j(\partial_ju^\epsilon_R\,w^\frac{\alpha}{2}\xi^2)$ so that
$\partial_i\phi=\partial_j(\partial_i\partial_ju^\epsilon_R w^\frac{\alpha}{2} \xi^2+\frac{\alpha}{2}\partial_iw \, w^\frac{\alpha-2}{2}\,\partial_ju^\epsilon_R\, \xi^2) + 2\partial_j(\xi\partial_i\xi\, w^\frac{\alpha}{2}\,\partial_ju^\epsilon_R)$. Sum in $j$ to get
\begin{equation*}
\begin{split}
A+B:&= \sum_{i,j=1}^2\int_{B_1} a_i(\partial_iu^\epsilon_R)\partial_j(\partial_i\partial_ju^\epsilon_R w^\frac{\alpha}{2} \xi^2+\frac{\alpha}{2}\partial_iw \, w^\frac{\alpha-2}{2}\,\partial_ju^\epsilon_R\, \xi^2)\dx \\
&+2\sum_{i,j=1}^2\int_{B_1}  a_i(\partial_iu^\epsilon_R) \partial_j(\xi\partial_i\xi\, w^\frac{\alpha}{2}\,\partial_ju^\epsilon_R)\dx=0.
\end{split}
\end{equation*}
Note that $\partial_i w=2\sum_{j=1}^2\partial_i\partial_ju^\epsilon_R\,\partial_ju^\epsilon_R$ and $\partial_ia_i(\partial_iu^\epsilon_R)\geq c_pw^\frac{p-2}{2}$ since $1<p<2$. Integrate by parts in $A$. We get $A=A_1+A_2$ where
\begin{equation*}
\begin{split}
A_1&:=\sum_{i,j=1}^2\int_{B_1}\partial_ia_i(\partial_iu^\epsilon_R)(\partial_i\partial_ju^\epsilon_R)^2\,w^\frac{\alpha}{2}\,\xi^2\dx
\geq c_p\sum_{j=1}^2 \int_{B_1} w^\frac{p-2+\alpha}{2}|\nabla\partial_ju^\epsilon_R|^2\xi^2\dx,\\
A_2&:=c\alpha\sum_{i,j=1}^2\int_{B_1}\partial_ia_i(\partial_iu^\epsilon_R)\partial_i\partial_ju^\epsilon_R\,\partial_ju^\epsilon _R\,\partial_iw\, w^\frac{\alpha-2}{2}\,\xi^2\dx
=c\alpha\sum_{i=1}^2 \int_{B_1}\partial_ia_i(\partial_iu^\epsilon_R) (\partial_i w)^2 w^\frac{\alpha-2}{2}\,\xi^2\dx\\
&\geq c_p\alpha \int_{B_1}w^\frac{p-4+\alpha}{2}|\nabla w|^2\xi^2\dx.
\end{split}
\end{equation*}
Now we estimate $B=B_1+B_2+B_3$. 
\begin{equation*}
\begin{split}
B_1:&=\sum_{i,j=1}^2\int_{B_1}a_i(\partial_iu^\epsilon_R)w^\frac{\alpha}{2}|\partial_ju^\epsilon_R|\,|\partial_j(\xi\partial_i\xi)|\dx
\leq C_p\int_{B_1}w^\frac{p+\alpha}{2} (|\nabla\xi|^2+|\nabla^2\xi|)\dx,\\
B_2:&=\frac{\alpha}{2}\sum_{i,j=1}^2\int_{B_1} a_i(\partial_iu^\epsilon_R)w^\frac{\alpha-2}{2}|\partial_jw|\,|\partial_ju^\epsilon_R| \,\xi\,|\partial_i\xi|\dx 
\leq C\alpha\int_{B_1} w^\frac{p+\alpha-2}{2}|\nabla w|\,\xi\,|\nabla\xi|\dx \\
&\leq \eta  \alpha \int_{B_1}w^\frac{p-4+\alpha}{2}|\nabla w|^2\xi^2\dx+\frac{C \alpha}{\eta}\int_{B_1}|\nabla\xi|^2\,w^\frac{p+\alpha}{2}\dx,\\
B_3:&=\sum_{i,j=1}^2\int_{B_1}a_i(\partial_iu^\epsilon_R)w^\frac{\alpha}{2}|\partial_j\partial_ju^\epsilon_R|\,|\partial_ju^\epsilon_R| \,\xi\,|\partial_i\xi|\dx 
\leq \sum_{j=1}^2\int_{B_1}w^\frac{p-1+\alpha}{2}|\nabla\partial_ju^\epsilon_R|\,\xi\,|\nabla\xi|\dx\\
&\leq \eta\sum_{j=1}^2 \int_{B_1} w^\frac{p-2+\alpha}{2}|\nabla\partial_ju^\epsilon_R|^2\xi^2\dx+\frac{C}{\eta}\int_{B_1}|\nabla\xi|^2\, w^\frac{p+\alpha}{2}\dx
\end{split}
\end{equation*}
where we used $a_i(\partial_iu^\epsilon_R)\leq w^\frac{p-1}{2}$ and Young's inequality with a parameter $\eta$ to be chosen suitably small. We get
\begin{equation}\label{mixEst2}
c_p\sum_{j=1}^2 \int_{B_1} w^\frac{p-2+\alpha}{2}|\nabla\partial_ju^\epsilon_R|^2\xi^2\dx
+ c_p\alpha \int_{B_1}w^\frac{p-4+\alpha}{2}|\nabla w|^2\xi^2\dx
\leq C_p(\alpha+1)\int_{B_1} (|\nabla\xi|^2+|\nabla^2\xi|)\, w^\frac{p+\alpha}{2}\dx.
\end{equation}
Note that for $\alpha=0$ we get for all $j\in\{1,2\}$
\begin{equation}\label{gradientEst2}
\int_{B_1} w^\frac{p-2}{2}|\nabla\partial_ju^\epsilon_R|^2\xi^2\dx
\leq C_p\int_{B_1} (|\nabla\xi|^2+|\nabla^2\xi|)\, w^\frac{p}{2}\dx,
\end{equation}
and since $|\nabla w|^2\leq c\sum_j |\nabla \partial_j u^\epsilon_R|^2|\nabla u^\epsilon_R|^2$ we have
\begin{equation}\label{alpha0}
\begin{split}
\int_{B_1}w^\frac{p-4}{2}|\nabla w|^2\xi^2\dx
&\leq c \sum_{j=1}^2\int_{B_1}w^\frac{p-4}{2}|\nabla u^\epsilon_R|^2|\nabla \partial_j u^\epsilon|^2\xi^2\dx
\leq c \sum_{j=1}^2\int_{B_1}w^\frac{p-2}{2}|\nabla \partial_j u^\epsilon_R|^2\xi^2\\
&\leq C_p\int_{B_1} (|\nabla\xi|^2+|\nabla^2\xi|)\, w^\frac{p}{2}\dx.
\end{split}
\end{equation}
Now for $\alpha\geq 1$, \eqref{mixEst2} implies
\begin{equation}
 \int_{B_1}w^\frac{p-4+\alpha}{2}|\nabla w|^2\xi^2\dx
\leq C_p \frac{\alpha+1}{\alpha}\int_{B_1} (|\nabla\xi|^2+|\nabla^2\xi|)\, w^\frac{p+\alpha}{2}\dx
\end{equation}
and combining with \eqref{alpha0} we get
\begin{equation*}
 \int_{B_1}\lvert\nabla (w^\frac{p+\alpha}{4}\xi)\rvert^2\dx
\leq C (p+\alpha)^2\int_{B_1} (|\nabla\xi|^2+|\nabla^2\xi|)\, w^\frac{p+\alpha}{2}\dx
\end{equation*}
for all $\alpha\geq 0$.
Using Sobolev's embedding $W_0^{1,2}(B_1)\hookrightarrow L^{2q}(B_1)$ for a fixed $q>1$ we get
\begin{equation}\label{moser2}
 \left(\int_{B_1} w^{q\frac{p+\alpha}{2}}\xi^{2q}\dx\right)^\frac{1}{q}
\leq C_p (p+\alpha)^2\int_{B_1} (|\nabla\xi|^2+|\nabla^2\xi|)\, w^\frac{p+\alpha}{2}\dx.
\end{equation} 
Now choose a sequence of radii $r_i=1/2^i+(1-1/2^{i})\frac{1}{2}$, cut-off functions $\xi$ between $r_i$ and $r_{i+1}$ and $\alpha_i=q^ip-p$ so that $\frac{p+\alpha_i}{2}=\frac{p}{2}q^i$. Using these in \eqref{moser2}, raising to the power $1/q^{i}$ and iterating  we get for all $i\in\mathbb{N}$
\begin{equation*}
\left(\int_{B_{r_{i+1}}} w^{\frac{p}{2}q^{i+1}}\dx \right)^\frac{1}{q^{i+1}}
\leq (C_p q^{2i} 2^i)^\frac{1}{q^i}\left( \int_{B_{r_i}} w^{\frac{p}{2}q^i}\dx\right)^\frac{1}{q^i}
\leq \prod_{j=0}^i(C_p q^{2j} 2^j)^\frac{1}{q^j} \int_{B_{1}} w^{\frac{p}{2}}\dx.
\end{equation*} 
Observe that $\prod_{i=0}^\infty(C_p q^{2i} 2^i)^\frac{1}{q^i}=C(p,q)<\infty$ so passing to the limit as $i\to\infty$ we get
$$\sup_{B_{1/2}}w^\frac{p}{2}\leq C(p,q)\int_{B_{1}} w^{\frac{p}{2}}\dx$$
which, after rescaling, proves \eqref{lip2}. Now going back to \eqref{gradientEst2}, choosing a cut-off function between $B_{R/2}$ and $B_R$ and using $1<p<2$ we get
\begin{equation*}
\int_{B_{R/2}} |\nabla \partial_j u^\epsilon|^2\dx \leq 
C_p\sup_{B_{R/2}} (\epsilon+|\nabla u^\epsilon|^2)^\frac{2-p}{p} \fint_{B_R}(\epsilon+|\nabla u^\epsilon|^2)^\frac{p}{2}\dx.
\end{equation*}
Using \eqref{lip2} and \eqref{energy2} we obtain \eqref{grad2}.
\end{proof}

Next we collect some facts about the convergence of $u^\epsilon$ to the solution of the degenerate equation. These are sufficient for our purposes.
\begin{Prop}\label{uni2}
Let $u^\epsilon$ be the solution of \eqref{orthonondeg2} for $1<p<2$ and $u\in W^{1,p}(\Omega)$ the solution of \eqref{orthodeg}. We have
\begin{itemize}
\item $u^\epsilon$ converges to $u$ locally uniformly in $B_R$,\\
\item $\nabla u^\epsilon$ converges to $\nabla u$ in $L^p(B_R)$.
\end{itemize}
\end{Prop}

\begin{proof}
From the energy estimate \eqref{energy2} we obtain a uniform bound for the $L^p$ norm of $\nabla u^\epsilon$. Therefore (up to a subsequence) $u^\epsilon$ converges to some $v\in W^{1,p}(B_R)$ weakly in $W^{1,p}(B_R)$ and strongly in $L^p(B_R)$. Note that we have $v-u\in W_0^{1,p}(B_R)$. By weakly lower semicontinuity we get
\begin{equation*}
\begin{split}
I_{B_R}(v) = \sum_{i=1}^2\int_{B_R}\frac{|\partial_i v|^p}{p}\dx
&\leq \liminf_{\epsilon\to 0} \sum_{i=1}^2\int_{B_R}\frac{|\partial_i u^\epsilon|^p}{p}\dx\\
&\leq \liminf_{\epsilon\to 0} \sum_{i=1}^2\int_{B_R}\frac{1}{p}(|\partial_i u^\epsilon|^2+\epsilon)^\frac{p}{2}\dx\\
&\leq \liminf_{\epsilon\to 0} \sum_{i=1}^2\int_{B_R}\frac{1}{p}(|\partial_i u|^2+\epsilon)^\frac{p}{2}\dx\\
&=\sum_{i=1}^2\int_{B_R}\frac{1}{p}|\partial_i u|^p\dx
=I_{B_R}(u).
\end{split}
\end{equation*}
Note that in the third inequality we used the minimality of $u^\epsilon$ subject to the boundary condition $u^\epsilon-u\in W_0^{1,p}(B_R)$. By uniqueness of the minimizer of $I_{B_R}$ among functions with boundary values $u$ in $B_R$, we get $v=u$. 
By the uniform Lipschitz estimate \eqref{lip2} and Ascoli-Arzela' theorem we obtain that the convergence is uniform.\\

Now we show $L^p(B_R)$ convergence of the gradient. Use $\phi = u^\epsilon-u$ as a test function in \eqref{orthonondeg2}, add and subtract the term $(|\partial_i u|^2+\epsilon)^\frac{p-2}{2}\partial_i u$ to get
\begin{equation*}
\begin{split}
\sum_{i=1}^2\int_{B_R} &\left((|\partial_i u^\epsilon|^2+\epsilon)^\frac{p-2}{2}\partial_i u^\epsilon-(|\partial_i u|^2+\epsilon)^\frac{p-2}{2}\partial_i u\right)\left(\partial_i u^\epsilon-\partial_i u\right)\dx\\
&=\sum_{i=1}^2\int_{B_R} (|\partial_i u|^2+\epsilon)^\frac{p-2}{2}\partial_i u (\partial_i u-\partial_i u^\epsilon)\dx.
\end{split}
\end{equation*}
Since  $\partial_i u-\partial_i u^\epsilon$  converges to $0$ weakly in $L^p(B_R)$, the integral in the right hand side converges to $0$. We can minorize the integral in the left hand side using the inequality 
$$ |a-b|^2(\epsilon+|a|^2+|b^2|)^\frac{p-2}{2}\leq C_p ((\epsilon+|a|^2)^\frac{p-2}{2}a-(\epsilon+|b|^2)^\frac{p-2}{2}b)(a-b) $$
valid for $1<p<2$, and obtain that
\begin{equation}\label{conv0}
\int_{B_R}\left(\epsilon+|\partial_i u^\epsilon|^2+|\partial_i u|^2\right)^\frac{p-2}{2}|\partial_i u^\epsilon-\partial_i u|^2\dx \longrightarrow 0
\end{equation} 
as $\epsilon\to 0$, for $i=1$, $2$.
Finally by H\"older's inequality
\begin{equation*}
\begin{split}
\int_{B_R} |\partial_iu^\epsilon-\partial_iu|^p\dx
&= \int_{B_R}  |\partial_iu^\epsilon-\partial_iu|^p \left(\epsilon+|\partial_i u^\epsilon|^2+|\partial_i u|^2\right)^\frac{p(p-2)}{2}\left(\epsilon+|\partial_i u^\epsilon|^2+|\partial_i u|^2\right)^\frac{p(2-p)}{2}\dx\\
&\leq \left( \int_{B_R} |\partial_iu^\epsilon-\partial_iu|^2 \left(\epsilon+|\partial_i u^\epsilon|^2+|\partial_i u|^2\right)^\frac{p-2}{2}\dx\right)^\frac{p}{2}\cdot\\
&\qquad\qquad\qquad\quad\cdot\left( \int_{B_R} \left(\epsilon+|\partial_i u^\epsilon|^2+|\partial_i u|^2\right)^\frac{p}{2}\dx\right)^\frac{2-p}{2}.
\end{split}
\end{equation*}
Since the last integral is uniformly bounded in $\epsilon$, using \eqref{conv0} we get  that $\partial_i u^\epsilon$ converges to $\partial_i u$ in $L^p(B_R)$.
\end{proof}

\section{Monotone functions and Lebesgue's lemma}
A continuous function $v:\Omega\longrightarrow\mathbb{R}$ is monotone (in the sense of Lebesgue) if
$$\max_{\overline{D}}v=\max_{\partial D}v\quad \text{and}\quad  \min_{\overline{D}}v=\min_{\partial D}v$$
for all subdomains $D\subset\subset \Omega$. Monotone functions are further discussed in \cite{Manf}.

The next Lemma is due to Lebesgue \cite{L}.
\begin{Lemma}\label{oscLeb}
Let $B_R\subset \mathbb{R}^2$ and  $v\in C(B_R )\cap W^{1,2}(B_R)$ be monotone in the sense of Lebesgue. Then
\begin{equation*}
(\osc{ B_r} v)^2\log\left(\frac{R}{r}\right)
\leq  \pi \int_{B_R\setminus B_r} |\nabla v(x)|^2 \dx
\end{equation*}
for every $r<R$.
\end{Lemma}

\begin{proof}
Assume $v$ is smooth. Let $(\eta, \zeta)$ be the center of $B_R$. Let $x_1$ and $x_2$ be two points on the circle of radius $t$, and let $\gamma:[0,2\pi]\longrightarrow \mathbb{R}^2$, $\gamma(s)=(\eta+t\cos(s),\zeta+ t\sin(s))$ be a parametrization of the circle such that $\gamma(a)=x_1$ and $\gamma(b)=x_2$. Then we have
\begin{equation*}
\begin{split}
v(x_1)-v(x_2)=\int_a^b \frac{d}{ds} v(\gamma(s)) \ds 
= \int_a^b \langle  \nabla v(\gamma(s)), \gamma '(s) \rangle \ds
\leq \int_a^{b} t\, |\nabla v(\gamma(s))| \ds.
\end{split}
\end{equation*}
Taking the supremum on angles $a$ and $b$ such that $|a-b|\leq \pi$ and using H\"older's inequality, we get
\begin{equation*}
(\osc{\partial B_t}v )^2 
\leq \pi t^2  \int_0^{2\pi}  |\nabla v (\gamma(s))|^2 \ds.
\end{equation*}
Now diving by $t$, integrating from $r$ to $R$, and using polar coordinates we get
\begin{equation*}
\int_r^R\frac{(\text{osc}_{\partial B_t}v )^2}{t} \dt
\leq \pi \int_r^R \int_0^{2\pi} t\,|\nabla v(\gamma(s))|^2 \ds \dt
= \pi \int_{B_R\setminus B_r} |\nabla v(x)|^2 \dx.
\end{equation*}
Thanks to the monotonicity of $v$, for $t\geq r$ we have
\begin{equation*}
\osc{\partial B_t} \partial_ju^\epsilon \geq
   \osc{B_t} \partial_ju^\epsilon
\geq \osc{B_r} \partial_ju^\epsilon
\end{equation*}
and we get the result for a smooth function. The general statement follows by approximation.
\end{proof}

The following is credited to \cite{BB} (see Lemma 2.14 for the minimum principle).
\begin{Lemma}\label{maxmin2}[Minimum and Maximum principles for the derivatives]\\
Let $u^\epsilon$ be the solution of \eqref{orthonondeg2}. Then 
$$\min_{\partial B_r}\partial_j u^\epsilon\leq \partial_j u^\epsilon (x)\leq\max_{\partial B_r}\partial_j u^\epsilon $$
for all $x\in B_r$, $B_r\subset\subset B_R$ and $j=1$, $2$. In particular, $\partial_j u^\epsilon$ is monotone in the sense of Lebesgue.

\end{Lemma}
\begin{proof}
We are going to show that given a constant $C$, if $\partial_j u^\epsilon \leq C$ (resp. $\partial_j u^\epsilon \geq C$) in $\partial B_r$ then $\partial_j u^\epsilon \leq C$ (resp. $\partial_j u^\epsilon \geq C$) in $B_r$.
Let $\phi^{\pm}= 1_{B_r}(\partial_j u^\epsilon-C)^\pm =1_{B_r}\max\{\pm(\partial_j u^\epsilon-C),0\} $ in the equation satisfied by the derivative \eqref{orthoder2}. Since $u^\epsilon$ is smooth and $\partial_j u^\epsilon \geq C$  (resp. $\partial_j u^\epsilon \leq C$) on $\partial B_r$ we have $\phi^\pm\in W_0^{1,2}(\Omega)$, so they are admissible functions. We get
\begin{equation*}
\begin{split}
0&=\sum_{i=1}^2 \int_{B_r} (\epsilon+|\partial_i u^\epsilon|^2)^\frac{p-4}{2} (\epsilon+(p-1)|\partial_i u^\epsilon|^2)\, |\partial_i (\partial_j u^\epsilon -C)^\pm |^2 \dx \\
&\geq \epsilon \sum_{i=1}^2\int_{B_r}(\epsilon+|\nabla u^\epsilon|^2)^\frac{p-4}{2}  |\partial_i (\partial_j u^\epsilon -C)^\pm |^2 \dx \\
&= \epsilon \int_{B_r}(\epsilon+|\nabla u^\epsilon|^2)^\frac{p-4}{2}  |\nabla (\partial_j u^\epsilon -C)^\pm |^2 \dx .
\end{split}
\end{equation*}
This implies $(\partial_j u^\epsilon -C)^\pm$ is constant in $B_r$, and since it is $0$ in $\partial B_r$ then $(\partial_j u^\epsilon -C)^\pm=0$ in $B_r$.
\end{proof}
\

\section{Proof of the Main Theorem}
\begin{proof}
[Proof of Theorem \eqref{C1}]
Applying Lemma \eqref{oscLeb} and estimate \eqref{grad2} we get for all $r<R/2$
\begin{equation}
(\osc{B_r} \partial_ju^\epsilon)^2 \log(\frac{R}{r}) \leq C \norm{\nabla \partial_ju^\epsilon}^2_{L^2(B_{R/2})} \leq C \left(\fint_{B_{R}} |\nabla u|^p\dx+\epsilon^\frac{p}{2}\right)^\frac{2}{p}
\end{equation}
and hence for all $r<R/2$
\begin{equation}\label{osc2}
\osc{B_r} \partial_ju^\epsilon  \leq C \left(\log\left(\frac{R}{r}\right)\right)^{-\frac{1}{2}}
\left(\fint_{B_{R}} |\nabla u|^p\dx+\epsilon^\frac{p}{2}\right)^\frac{1}{p}
\end{equation}
where $C$ is a constant independent of $\epsilon$.

Thanks to Proposition \eqref{uni2} we can pass to the limit and get \eqref{oscEst2}.
\end{proof}

\subsection{Acknowledgements}
I thank Peter Lindqvist for useful comments and suggestions.

\bibliography{reference_bib}\nocite{*}
\bibliographystyle{abbrv}

\end{document}